\documentclass[12pt, reqno]{amsart}
\usepackage{fullpage,url,amssymb}
\usepackage{hyperref}
\usepackage[alphabetic,lite]{amsrefs}

\usepackage{amscd}   % for commutative diagrams
\usepackage[all, cmtip]{xy} % for complicated commutative diagrams

\usepackage{color}
\newtheorem{lemma}{Lemma}[section]
\newtheorem{theorem}[lemma]{Theorem}

\newtheorem{claim*}{Claim}

\newtheorem{remark}[lemma]{Remark}
\newtheorem{thm}[lemma]{Theorem}

\newcommand{\G}{{\mathbb G}}

\newcommand{\PP}{{\mathbb P}}

\newcommand{\Q}{{\mathbb Q}}

\newcommand{\Z}{{\mathbb Z}}

\newcommand{\Kbar}{{\overline{K}}}
\newcommand{\Lbar}{{\overline{L}}}

\newcommand{\Xbar}{{\overline{X}}}

\newcommand{\kk}{{\mathbf k}}

\newcommand{\Ktilde}{\widetilde{K}}
\newcommand{\Ltilde}{\widetilde{L}}

\newcommand{\calA}{{\mathcal A}}

\DeclareMathOperator{\im}{im}

\DeclareMathOperator{\Ind}{Ind}
\DeclareMathOperator{\Cor}{Cor}

\DeclareMathOperator{\Norm}{Norm}
\DeclareMathOperator{\Br}{Br}

\DeclareMathOperator{\divv}{div}

\DeclareMathOperator{\Pic}{Pic}
\DeclareMathOperator{\Jac}{Jac}

\DeclareMathOperator{\HH}{H}
\DeclareMathOperator{\N}{N}

\newcommand{\isom}{\cong}

\numberwithin{equation}{section}
\numberwithin{table}{section}

\newcommand{\defi}[1]{\textsf{#1}} % for defined terms

\title{Transcendental Brauer elements via descent on elliptic surfaces}

\author{Bianca Viray}
\thanks{This research was partially supported by NSF grants DMS-1002933 and DMS-0841321 and a Ford Foundation dissertation year fellowship.}

\address{Department of Mathematics, Box 1917, Brown University, Providence, RI
			02912, USA}
\email{bviray@math.brown.edu}
\urladdr{http://math.brown.edu/\~{}bviray}

\date{}

\begin{document}

	%%%%%%%%%%%%%%%%%%%%%%%%%%%%%%%%%%%%%%%%%%%%%%%%%%%%%%%%%%%%%%%%%%%%%%%%%%%%
	\begin{abstract}
		Transcendental Brauer elements are notoriously difficult to compute.  Work of Wittenberg, and later, Ieronymou, gives a method for computing $2$-torsion transcendental classes on surfaces that have a genus $1$ fibration with rational $2$-torsion in the Jacobian fibration.  We use ideas from a descent paper of Poonen and Schaefer to remove this assumption on the rational $2$-torsion.
	\end{abstract}
	%%%%%%%%%%%%%%%%%%%%%%%%%%%%%%%%%%%%%%%%%%%%%%%%%%%%%%%%%%%%%%%%%%%%%%%%%%%%

	\maketitle

	%%%%%%%%%%%%%%%%%%%%%%%%%%%%%%%%%%%%%%%%%%%%%%%%%%%%%%%%%%%%%%%%%%%%%%%%%%%%
	\section{Introduction}%%%%%%%%%%%%%%%%%%%%%%%%%%%%%%%%%%%%%%%%%%%%%%%%%%%%%%
	%%%%%%%%%%%%%%%%%%%%%%%%%%%%%%%%%%%%%%%%%%%%%%%%%%%%%%%%%%%%%%%%%%%%%%%%%%%%
	
		Let $X$ be a smooth projective geometrically integral variety over a field $k$.  The \defi{Brauer group} of $X$, denoted $\Br X$, is the \'etale cohomology group $\HH_{\textup{et}}^2(X, \G_m)$; it can also be thought of as the \defi{unramified} part of $\Br \kk(X)$, where $\kk(X)$ is the function field of $X$.  This latter point of view is valuable because it gives us an explicit way to represent the \emph{elements} of $\Br X$, i.e., as unramified central simple algebras over $\kk(X)$.
		
		The Brauer group of $X$ always contains \defi{constant algebras} $\Br_0 X := \im\left(\Br k\to\Br X\right)$.
		These algebras encode little to no information about the arithmetic and geometry of $X$ so our study will focus on the quotient $\frac{\Br X}{\Br_0 X}$.  We are broadly motivated by the following problems.
		\begin{enumerate}
			\item Determine unramified central simple algebras$/\kk(X)$ that generate $\left(\frac{\Br X}{\Br_0 X}\right)[m]$.
			
			\item Given an unramified central simple algebra$/\kk(X)$, determine if it is trivial in $\frac{\Br X}{\Br_0 X}$.
		\end{enumerate}

	Solutions to the above problems have arithmetic applications.  If $k$ is a global field, then, by work of Manin~\cite{Manin-BMobs}, elements of the Brauer group can obstruct the existence of $k$-points, even if there exist local points, i.e. $k_v$-points for every completion $k_v$ of $k$.  Computation of this obstruction requires \emph{explicit} representations of the elements of $\frac{\Br X}{\Br_0 X}$; knowledge of the group structure of $\frac{\Br X}{\Br_0 X}$ does not suffice.

		From an arithmetic point of view, these elements are further subdivided: the \defi{algebraic} elements are those contained in $\Br_1 X := \ker(\Br X \to \Br \Xbar)$, and the \defi{transcendental} elements are those that give non-trivial classes in $\frac{\Br X}{\Br_1 X}$.  Many methods have been developed for finding explicit representatives for algebraic elements, at least for elements that are $2$ and $3$-torsion.  Most of these methods make heavy use of a spectral sequence which (under mild assumptions) gives the following isomorphism
		\begin{equation}\label{eqn:HS-Brauer-isom}
			\frac{\Br_1 X}{\Br_0 X} \stackrel{\sim}{\longrightarrow}
			\HH^1(\G_k, \Pic \Xbar).
		\end{equation}
		In contrast, very little is known about the computation of transcendental elements.  
		
		Indeed, only a handful of papers have addressed the above problems in the case of transcendental Brauer elements~\cite{Harari-transcendental, Wittenberg-transcendental, SSD-2descent, HS-Enriques, Ieronymou-transcendental, KT-effectivity, HVAV-K3, HVA-K3Hasse}.  The work of Kresch and Tschinkel is arguably the most general.  They show that there is an effective solution to the above problems for surfaces whose geometric Picard group is finitely generated and torsion free.  A large class of surfaces satisfy these assumptions, including general K3 surfaces, but there are some  exceptions, e.g. Enriques surfaces.
		
		In this paper, we focus on surfaces with a genus $1$ fibration.  Although general K3 surfaces have no such structure, Enriques surfaces are always equipped with a genus $1$ fibration.  Thus the cases we consider are neither a superset nor a subset of those studied in~\cite{KT-effectivity}.
		
		We build on two papers that explicitly compute transcendental $2$-torsion elements in the Brauer group of an elliptic K3 surface~\cite{Wittenberg-transcendental, Ieronymou-transcendental}.  Although Wittenberg and Ieronymou each focus on a particular K3 surface, we expect that their methods can be applied to compute $2$-torsion transcendental Brauer elements on any surface with a genus $1$ fibration such that
		% In this paper, we build on two papers that explicitly compute transcendental $2$-torsion elements in the Brauer group of an elliptic K3 surface~\cite{Wittenberg-transcendental, Ieronymou-transcendental}.  Although Wittenberg and Ieronymou each focus on a particular K3 surface, we expect that their methods can be applied to compute $2$-torsion transcendental Brauer elements on any surface with a genus $1$ fibration such that 
		\begin{enumerate}
			\item the fibration has a section over a quadratic extension, and
			\item the Jacobian fibration has rational $2$-torsion.
		\end{enumerate} 
		Their methods echo the classical descent methods used to compute the rank of an elliptic curve.  More precisely, they use explicit elements of the cohomology group $\HH^1(G_K, J[2])$ to construct Brauer elements on the generic fiber, and then determine which elements spread out to Brauer elements on the surface.

		If the $2$-torsion of the Jacobian fibration is \emph{not} rational, then it is difficult to compute the elements of $\HH^1(G_K, J[2])$.  This difficulty also arises in a more classical number theory problem, i.e., that of computing the rank of a Jacobian over a number field, and many papers have studied various methods of working around this difficulty. 
		
		We take our inspiration from a paper of Poonen and Schaefer~\cite{PS-descent}.  Instead of studying the full $\HH^1(G_K, J[2])$, we study a quotient of $\HH^1(G_K, J[2])$ and use elements of this quotient to construct Brauer elements of the generic fiber.  We then give a necessary condition for the elements to spread out to Brauer elements of the surface.  This condition rules out all but finitely many elements of the quotient.

		To give a precise statement of our result, we fix some notation.  Let $X$ be a smooth projective geometrically integral surface over an algebraically closed field of characteristic $0$ with $\pi\colon X \to W$ a genus $1$ fibration.  Let $K$ denote the function field of $W$, $C$ denote the generic fiber of $\pi$ and $J$ denote the Jacobian of $C$.  Recall that $\Br X\subseteq \Br C$.  

		\begin{theorem}\label{thm:main}
			Assume that $C$ has a model $y^2 = f(x)$ where $\deg(f) = 4$ and let $L$ be the degree $4$ \'etale algebra $K[\alpha]/(f(\alpha))$.
			{Then the following diagram commutes,}
			\begin{equation}\label{eqn:main}
				\xymatrix{
				\frac{J(K)}{2J(K)} \ar @{^{(}->}[r]  \ar [rd]^{x - \alpha}
				& \HH^1(G_K, J[2])\ar @{->>}[r] {{\ar@{->>}[d]}}
				& \HH^1(G_K, J)[2] \ar[d]\\
				&
				{{\ker \left(N:\frac{L^{\times}}{L^{\times2}K^{\times}} 
					\to \frac{K^{\times}}{K^{\times2}}\right)\ar[r]^-h}}
				&(\Br C)[2],
				}
			\end{equation}
			where $h\colon \ell \mapsto \Cor_{\kk(C_L)/\kk(C)}\left( (\ell, x - \alpha)_2\right)$.
			
			 Moreover, there is a finite set $\left(\ker N\right)_{S\textup{-unr}}$ such that 
			\[
				h^{-1}(\Br X) \subseteq \left(\ker N\right)_{S\textup{-unr}}.
			\]
		\end{theorem}
			
			The finite set $\left(\ker N\right)_{S\textup{-unr}}$ should be thought of as a set which agrees with a ``fake $2$-Selmer group'' away from a set $S$ of finitely many places.  The idea of studying the Brauer group of an elliptic surface via a ``Selmer-like'' object is not new.  Colliot-Th\'el\`ene, Skorobogatov, and Swinnerton-Dyer defined a \defi{geometric Selmer group} for any Jacobian fibration~\cite{CTSkorSD-GeometricSelmer}.  
			
			In the case of classical descent on elliptic curves, typically the $n$-Selmer group is useful theoretically, but difficult to explicitly compute for a generic elliptic curve.  In practice, to obtain an upper bound on the rank of an elliptic curve, one tends to use a group that agrees with the Selmer group away from finitely many places.  We expect that an analogous trade-off occurs with the geometric Selmer group of~\cite{CTSkorSD-GeometricSelmer} and the finite set $\left(\ker N\right)_{S\textup{-unr}}$ that we study in this paper.

		%%%%%%%%%%%%%%%%%%%%%%%%%%%%%%%%%%%%%%%%%%%%%%%%%%%%%%%%%%%%%%%%%%%%%%%%
		\subsection*{Outline}%%%%%%%%%%%%%%%%%%%%%%%%%%%%%%%%%%%%%%%%%%%%%%%%%%%
		%%%%%%%%%%%%%%%%%%%%%%%%%%%%%%%%%%%%%%%%%%%%%%%%%%%%%%%%%%%%%%%%%%%%%%%%
			In \S\ref{sec:cohomology}, we show that cohomological arguments produce the commutative diagram in~\eqref{eqn:main}.  These cohomological arguments do not result in an explicit description of the map $h$; we prove that $h$ has the description given above in~\S\ref{sec:def_h}.  Finally, in~\S\ref{sec:Selmer}, we define $\left(\ker N\right)_{S\textup{-unr}}$, prove that this set is finite, and that $h^{-1}(\Br X) \subseteq \left(\ker N\right)_{S\textup{-unr}}$.  
			
		%%%%%%%%%%%%%%%%%%%%%%%%%%%%%%%%%%%%%%%%%%%%%%%%%%%%%%%%%%%%%%%%%%%%%%%%
		\subsection*{Notation}%%%%%%%%%%%%%%%%%%%%%%%%%%%%%%%%%%%%%%%%%%%%%%%%%%
		%%%%%%%%%%%%%%%%%%%%%%%%%%%%%%%%%%%%%%%%%%%%%%%%%%%%%%%%%%%%%%%%%%%%%%%%
			Throughout, $X$ will denote a smooth projective geometrically integral surface over an algebraically closed field of characteristic $0$.  We assume that there exists a morphism $\pi\colon X\to W$ such that the generic fiber is a smooth genus $1$ curve.  Let $K$ denote the function field of $W$; note that $K$ is a $C_1$ field.  Let $C$ denote the generic fiber of $\pi$.  We write $J$ for the Jacobian of $C$.  As stated in Theorem~\ref{thm:main}, we will assume that $C$ has a model of the form $y^2 = f(x)$ where $\deg(f) = 4$.  
			
			For any field $k$, $G_k$ will denote the absolute Galois group of $k$, and  $\HH^i(k, M)$ will denote group cohomology $\HH^i(G_k, M)$.  For $V$ a variety we write $\kk(V)$ for the function field of $V$.  For any valuation $v$ on $\kk(V)$, we will denote the residue field of $v$ by $\kappa(v)$.

	%%%%%%%%%%%%%%%%%%%%%%%%%%%%%%%%%%%%%%%%%%%%%%%%%%%%%%%%%%%%%%%%%%%%%%%%%%%%
	\section*{Acknowledgements}%%%%%%%%%%%%%%%%%%%%%%%%%%%%%%%%%%%%%%%%%%%%%%%%%
	%%%%%%%%%%%%%%%%%%%%%%%%%%%%%%%%%%%%%%%%%%%%%%%%%%%%%%%%%%%%%%%%%%%%%%%%%%%%
		I thank Bjorn Poonen for suggesting the problem.  I also thank Jean-Louis Colliot-Th\'el\`ene, Bjorn Poonen, and Anthony V\'arilly-Alvarado for helpful conversations.
			
	%%%%%%%%%%%%%%%%%%%%%%%%%%%%%%%%%%%%%%%%%%%%%%%%%%%%%%%%%%%%%%%%%%%%%%%%%%%%
	\section{Obtaining the commutative diagram}\label{sec:cohomology}%%%%%%%%%%%
	%%%%%%%%%%%%%%%%%%%%%%%%%%%%%%%%%%%%%%%%%%%%%%%%%%%%%%%%%%%%%%%%%%%%%%%%%%%%
	
		In order to obtain the commutative diagram in Theorem~\ref{thm:main}, we must first understand the Galois cohomology of $\Lbar := L\otimes_K\Kbar$ and various related Galois modules.  
		
		Consider the short exact sequence $0 \to \Kbar^{\times} \to \Lbar^{\times} \to 
		\Lbar^{\times}/\Kbar^{\times} \to 0$.  The long exact sequence in cohomology begins
			\[
				0 \to   K^{\times} \to L^{\times} \to 
				\HH^0\left(K, \frac{\Lbar^{\times}}{\Kbar^{\times}}\right)
				\to \HH^1(K, \Kbar^{\times}) \to 
				\HH^1(K, \Lbar^{\times}) \to 
				\HH^1\left(K, \frac{\Lbar^{\times}}{\Kbar^{\times}}\right)
				\to \Br K.
			\]
		By Tsen's Theorem~\cite[Thm. 6.2.8]{GS-csa}, $\Br K = 0$ and by a generalized version of Hilbert $90$~\cite[p. 152, Ex. 2]{Serre-LocalFields}, $\HH^1(K, \Kbar^{\times}) = \HH^1(K, \Lbar^{\times}) = 0$. Therefore
			\begin{equation}\label{eq:cohom-LbaroverKbar}
				\HH^0(K, \Lbar^\times/\Kbar^\times) 
				\isom L^\times/K^\times, 
				\quad \textup{ and } 
				\HH^1(K, \Lbar^\times/\Kbar^\times) = 0.
			\end{equation}

		Now consider the short exact sequence
		\[
			0 \to \mu_2(\Lbar)/\mu_2(\Kbar) \to 
			\Lbar^{\times}/\Kbar^{\times} \to
			\Lbar^{\times}/\Kbar^{\times} \to 0,
		\]
		where the last non-trivial map is the squaring map.  After taking group cohomology and applying~\eqref{eq:cohom-LbaroverKbar}, we have
		\begin{equation}\label{eq:H1mu2}
			\HH^1(K, \mu_2(\Lbar)/\mu_2(\Kbar)) \isom 
			L^{\times}/L^{\times2}K^{\times}.
		\end{equation}

		Now we will relate the above cohomology groups to those of $J[2]$.  The Galois module $J[2]$ is generated by differences of certain points of $C$, namely the $P_i := (\alpha_i, 0)$ where $\alpha_i$ is a root of $f(x)$.  Using this fact, we can define a morphism of Galois modules
			\[
				J[2] \hookrightarrow \frac{\mu_2(\Lbar)}{\mu_2(\Kbar)}, 
				\quad P_i - P_j \mapsto e_ie_j,
			\]
		where $e_i$ is identified with the vector with a $-1$ in the $i^{\textup{th}}$ entry, and a $1$ in all other entries, under the isomorphism $\Lbar = L\otimes_K \Kbar \isom \Kbar^4$.  This injection fits into the exact sequence
		\[
			0 \to J[2] \longrightarrow \frac{\mu_2(\Lbar)}{\mu_2(\Kbar)} 
			\stackrel{\textup{Norm}}{\longrightarrow} \mu_2(\Kbar) \to 0.
		\]
		
		Now consider the following diagram, where each row is exact and $\gamma = \infty^+ + \infty^-$.
		\begin{equation}\label{diag:kummer-brauer-short}
			\xymatrix {
				0 \ar[r] 
				& J[2] \ar[r] \ar[d] 
				& \mu_2(\Lbar)/\mu(\Kbar) \ar[r]^-{\N} \ar[d]
				& \mu_2 \ar[r] \ar[d]
				& 0 \\
				0 \ar[r]
				& J \ar[r]
				& \Pic(\overline{C})/\Z\gamma \ar[r]^-{\deg}
				& \Z/2\Z \ar[r]
				& 0\\
				0 \ar[r]
				& J \ar[r] \ar[u]
				& \Pic(\overline{C}) \ar[r]^-{\deg} \ar@{->>}[u]
				& \Z \ar[r] \ar@{->>}[u]
				& 0
			}
		\end{equation}
		After taking group cohomology, taking the $2$-torsion parts, and using~\eqref{eq:H1mu2} we obtain the following commutative diagram, where each row is exact.
		\begin{equation}\label{diag:kummer-brauer-long}
			\xymatrix {
				& \mu_2 \ar[r] \ar[d] 
				& \HH^1(K, J[2]) \ar@{->>}[r] \ar@{->>}[d]
				& \ker (\N\colon L^{\times}/L^{\times2}K^{\times} \to 
					K^{\times}/K^{\times2})
					\ar[d]\\
				& \Z/2\Z \ar[r]
				& \HH^1(K, J)[2] \ar[r] 
				& (\ker (\deg \colon \HH^1(K,
				 \Pic(\overline{C})/\Z\gamma)
				 	\to K^{\times}/K^{\times2}))[2]\\
				& \Z/2\Z \ar[r] \ar@{->>}[u] 
				& \HH^1(K, J)[2] \ar[r] \ar[u]
				& (\Br C)[2]  \ar[u]
			}
		\end{equation}
		The bottom two leftmost arrows are isomorphisms, so we have 
		\[
			\im\left(\HH^1(K, J)[2]) \to (\ker \deg)[2]\right) \isom 
			\im\left(\HH^1(K, J)[2]) \to(\Br C)[2]\right). 
		\]
		By the commutativity of the top right square, and the surjectivity of both maps with domain $HH^1(K, J[2])$, the image of $\ker N \to (\ker \deg)[2]$ coincides with the image of $\HH^1(K, J)[2] \to (\ker \deg)[2]$.  We can now combine everything into the following diagram.
		\begin{equation}
			\xymatrix{
			& \mu_2 \ar[r] \ar[d]
			& \Z/2\Z \ar[d]
			\\
			\frac{J(K)}{2J(K)} \ar @{^{(}->}[r]  \ar [rd]^{x - \alpha}
			& \HH^1(K, J[2])\ar @{->>}[r] \ar @{->>}[d]
			& \HH^1(K, J)[2] \ar[d]\\
			&{{\ker \left(N:\frac{L^{\times}}{L^{\times2}K^{\times}} 
				\to \frac{K^{\times}}{K^{\times2}}\right)\ar[r]}}
			&(\Br C)[2]
			}
		\end{equation}
	
		\begin{remark}
			Note that we have not claimed that two top vertical arrows are injective.  In fact, for some curves $C$ both maps can be trivial.  By calculating the $0^{\textup{th}}$ cohomology of the top row in~\eqref{diag:kummer-brauer-short}, one can show that the map $\mu_2 \to \HH^1(K, J[2])$ is injective if and only if $f(x)$ has no roots in $K$.  Similarly, by using the bottom row of~\eqref{diag:kummer-brauer-short}, one can show that $\Z/2\Z \to \HH^1(K, J)[2]$ is injective if and only if there are no $0$-cycles of degree $1$ defined over $K$.
		\end{remark}
	
	%%%%%%%%%%%%%%%%%%%%%%%%%%%%%%%%%%%%%%%%%%%%%%%%%%%%%%%%%%%%%%%%%%%%%%%%%%%%
	\section{Explicit computations of cohomological maps}\label{sec:def_h}%%%%%%
	%%%%%%%%%%%%%%%%%%%%%%%%%%%%%%%%%%%%%%%%%%%%%%%%%%%%%%%%%%%%%%%%%%%%%%%%%%%%
		In this section, we show that the bottom horizontal map in
		\begin{equation*}
			\xymatrix{
			\frac{J(K)}{2J(K)} \ar @{^{(}->}[r]  \ar [rd]^{x - \alpha}
			& \HH^1(K, J[2])\ar @{->>}[r] {{\ar@{->>}[d]}}
			& \HH^1(K, J)[2] \ar[d]\\
			&
			{{\ker \left(N:\frac{L^{\times}}{L^{\times2}K^{\times}} 
				\to \frac{K^{\times}}{K^{\times2}}\right)\ar[r]}}
			&(\Br C)[2],
			}
		\end{equation*}
		obtained in~\S\ref{sec:cohomology} by cohomological methods agrees with the map $\ell \mapsto \Cor_{\kk(C_L)/\kk(C)}\left( (\ell, x - \alpha)_2\right).$
		
		\begin{lemma}
			The map
			\[
				\ker \left(\N\colon L^{\times}/L^{\times2}K^{\times} \to K^{\times}/K^{\times2}\right) \to \Br \kk(C), \quad \ell \mapsto \Cor_{\kk(C_L)/\kk(C)}\left( (\ell, x - \alpha)_2\right)
			\]
			is well-defined and its image is contained in $\Br C$.
		\end{lemma}
		\begin{proof}
			Throughout this proof $\Cor$ will denote $\Cor_{\kk(C_L)/\kk(C)}$, unless another field extension is specified.  Let $m \in K^{\times}$ and $\ell, \ell_0 \in L^{\times}$, and consider $\calA := \Cor\left( (\ell\cdot\ell_0^2\cdot m, x - \alpha)_2\right)$.  We may expand $\calA$ as follows
			\begin{align*}
				\calA & = \Cor\left( (\ell, x - \alpha)_2\right) 
					+ \Cor\left( (\ell_0^2, x - \alpha)_2\right) 
					+ \Cor\left( (m, x - \alpha)_2\right)\\
				& = \Cor\left( (\ell, x - \alpha)_2\right) 
					+ (m, \Cor(x - \alpha))_2\\
				& = \Cor\left( (\ell, x - \alpha)_2\right) +  (m, f(x))_2\\
				& = \Cor\left( (\ell, x - \alpha)_2\right) +  (m, y^2)_2 \quad 
				= \quad \Cor\left( (\ell, x - \alpha)_2\right).
			\end{align*}
			This proves that the map is well-defined.
			
			Now consider the following diagram:
			\begin{equation}\label{cube1}
			   \xymatrix {
			    K^M_2(\kk(C_L)) 
					\ar[dd]^{\N_{\kk(C_L)/\kk(C)}} 
					\ar[rr]^{\oplus \partial^M_w} 
					\ar[dr]^{h^2_{\kk(C_L),2}} & & 
				\bigoplus_{w|v}K^M_1(\kappa(w)) 		
					\ar[dd]|\hole^(.7){\sum_{w|v}\N_{\kappa(w)/\kappa(v)}} 	
					\ar[dr]^{\oplus h^1_{\kappa(w),2}}\\
			    & \HH^2(\kk(C_L),\mu_2^{\otimes2})  
					\ar[rr]^(.3){\oplus \partial^2_w} 
					\ar[dd]^(.7){\Cor_{\kk(C_L)/\kk(C)}} & & 
				{\bigoplus_{w|v}\HH^1(\kappa(w), \mu_2)} 
					\ar[dd]^{\sum_{w|v}\Cor_{\kappa(w)/\kappa(v)}}  \\
			    K^M_2(\kk(C))
					\ar[rr]^(.7){\partial^M_v}|\hole 
					\ar[dr]^{h^2_{k(X),2}} & & K^M_1(\kappa(v)) 	
					\ar[dr]^{h^1_{\kappa(v),2}}\\
			    & \HH^2(\kk(C), \mu_2^{\otimes2}) 
					\ar[rr]^{\partial^2_v}  & & 
				\HH^1(\kappa(v), \mu_2) 
			   }
			\end{equation}
			The back, side, top, and bottom squares are all commutative~\cite[Cor. 7.4.3 and Prop. 7.5.1 \& 7.5.5]{GS-csa}.  Therefore, all ways of traversing from $K^M_2(\kk(C_L))$ to $\HH^1(\kappa(v), \mu_2)$ are equivalent.  By the Merkurjev-Suslin theorem\cite[Thm. 4.6.6]{GS-csa}, $h^2_{\kk(C_L),2}$ is surjective so the front square is commutative.  Using the purity theorem\cite{Fujiwara-purity} and the commutativity of the front and right square of~\ref{cube1}, we obtain the following commutative diagram where the bottom row is exact.

			\begin{equation}\label{diag:purity-cores}
					\begin{CD}
						&&&&\Br \kk(C_L)[2] 
						@>>\oplus_w \partial^2_w> \bigoplus_{v} 
						\left(\bigoplus_{w\mid v} 
						\kappa(w)^{\times}/\kappa(w)^{\times2}\right)\\
						&&&&@VV\operatorname{Cor}_{\kk(C_L)/\kk(C)}V 
						@VV\prod_{w\mid v}
						{\operatorname{N}_{\kappa(w)/\kappa(v)}}V\\
						0 @>>> \Br C[2] @>>> \Br \kk(C)[2] 
						@>>\oplus_v \partial^2_v>
						\bigoplus_{v}\kappa(v)^{\times}/\kappa(v)^{\times2}
					\end{CD}
			\end{equation}
			(Here, the direct sum $\oplus_v$ ranges over all valuations corresponding to prime divisors on $C$, and $\oplus_{w|v}$ ranges over all valuations corresponding to prime divisors in $C_L$ lying over $v$.)
			
			Thus, to show that $\Cor\left((\ell, x - \alpha)_2\right)$ is in $\Br C$, it suffices to show that
			\begin{equation}\label{eqn:unramified}
				\prod_{w|v} \Cor_{\kappa(w)/\kappa(v)} \left((-1)^{w(\ell)w(x - \alpha)}
				\ell^{w(x- \alpha)} (x - \alpha)^{w(\ell)}\right)
			\end{equation}
			is a square in $\kappa(v)^{\times}$, for all valuations $v$.  Since $\ell$ is a constant in $\kk(C_L)$, the valuation $w(\ell) = 0$ for all $w$.  If $w(x - \alpha) = 0$ for all $w\mid v$ then~\eqref{eqn:unramified} is 1, and hence a square.  We restrict our attention to valuations $v$ such that there exists a $w\mid v$ with $w(x - \alpha)\neq 0$. 
			
			For all valuations $w$ such that $w(x - \alpha)$ is positive, we have that valuation $w(x - \alpha) \equiv 0 \pmod 2 $ so~\eqref{eqn:unramified} is clearly a square.  Now we consider the valuations $v$ for which there exists a $w\mid v$ with $w(x - \alpha) <0$.  These valuation(s) $v$ correspond to the points at infinity on $C$, and for every $w\mid v$ we have that $w(x - \alpha) = -1$.  In this case~\eqref{eqn:unramified} can be simplified to
			\[
				\prod_{w|v} \Cor_{\kappa(w)/\kappa(v)} \left((-1)^{w(\ell)w(x - \alpha)}
				\ell^{w(x- \alpha)} (x - \alpha)^{w(\ell)}\right) = \prod_{w|v} \Cor_{\kappa(w)/\kappa(v)} \left(\ell^{-1}\right) = \Norm_{L/K}(\ell^{-1}).
			\]
			Since, by assumption, $\ell \in \ker \left(\N\colon L^{\times}/L^{\times2}K^{\times} \to K^{\times}/K^{\times2}\right)$, this is a square.\end{proof}
			
		\begin{thm}
			For any $K$, $C$, as above, the following diagram commutes
			\begin{equation}\label{eq:diag}
				\xymatrix{
					\HH^1(K, J[2])\ar[r]\ar[d] & \HH^1(K, J)[2]\ar[d]\\
					\ker \left(\N\colon L^{\times}/L^{\times2}K^{\times} 
					\to K^{\times}/K^{\times2}\right) \ar[r]^(.7)h & \Br_1 C_K\\
				}
			\end{equation}
			where $h(\ell) = \Cor_{k(C_L)/k(C)}\left(\ell, x - \alpha\right).$
		\end{thm}
		\begin{proof}
			First we show that if $f(x)$ splits completely over $K$~\eqref{eq:diag} commutes.  After making a change of variables on $\PP^1_{(x:z)}$, we may assume that $f(x) = (1- a_1x)(1 - a_2x)(1 - a_3x)x$, for some $a_1, a_2, a_3 \in K$.  Under these assumptions $C$ has a Weierstrass model: 
			\[
				Y^2 = (X - a_1)(X - a_2)(X - a_3).
			\]
			By~\cite[\S4.1, 4.4, Exer. 1\&2]{Skorobogatov-torsors}, the composition of the top and right map $\HH^1(K, J[2]) \to \Br_1 C$ is given by $(d_1, d_2) \mapsto (d_1, X - a_1)_2 + (d_2, X - a_2)_2$, where $(d_1, d_2)$ corresponds to the cocycle $\sigma \mapsto \frac{\sqrt{d_1} - \sigma(\sqrt{d_1})}{\sqrt{d_1}}(a_1, 0) + \frac{\sqrt{d_2} - \sigma(\sqrt{d_2})}{\sqrt{d_2}}(a_2, 0)$.  By tracing through the maps given in~\S\ref{sec:cohomology}, we see that the bottom left map is given by
			\[
				(d_1, d_2) \mapsto (d_1, d_2, 1, d_1d_2) \mapsto 
				(d_1, x - 1/a_1)_2 + (d_2, x - 1/a_2)_2 + (1, x - 1/a_3)_2 + 
				(d_1d_2, x)_2
			\]
			Using the relation $x = 1/X$ and other relations in the Brauer group, we can rewrite 
			\begin{align*}
				& (d_1, x - 1/a_1)_2 + (d_2, x - 1/a_2)_2 + (1, x - 1/a_3)_2 + 
			(d_1d_2, x)_2\\
				& =  \left(d_1, \frac{X - a_1}{-a_1X}\right)_2 
				+ \left(d_2, \frac{X - a_2}{-a_2X}\right)_2 + (d_1d_2, X)_2\\
				& =  (d_1, X - a_1)_2 + (d_1, -a_1)_2 (d_1, X)_2 + 
				(d_2, X - a_2)_2 + (d_2, -a_2)_2 (d_2, X)_2 + (d_1d_2, X)\\
				& =  (d_1, X - a_1) + (d_2, X - a_2) + (d_1, -a_1) + (d_1, -a_2)
			\end{align*} 
			By Tsen's theorem,  $(d_1, -a_1), (d_1, -a_2)$ are trivial quaternion algebras.  Thus, if $f(x)$ is a product of linear factors then~\eqref{eq:diag} commutes.
								
			Now we prove the theorem in the general case.  Let $\Ktilde$ be the splitting field of $f(x)$.  We will show that if~\eqref{eq:diag} commutes over $\Ktilde$, then it commutes over $K$.  Consider the following diagram, where $\Ltilde := L\otimes_K\Ktilde$:
			\begin{equation}\label{cube}
			   \xymatrix {
			    \HH^1(\Ktilde, J[2]) \ar[dd] \ar[rr] \ar[dr]^{\Cor} & & 
				\HH^1(\Ktilde, J)[2] \ar[dd]|\hole \ar[dr]^{\Cor}\\
			    & \HH^1(K, J[2]) \ar[rr] \ar[dd] & & 
				\HH^1(K, J)[2] \ar[dd]  \\
		    	\ker \N_{\Ltilde/\Ktilde} \ar[rr]|\hole \ar[dr]^{\Cor} & & 
				\Br C_{\Ktilde}	\ar[dr]^{\Cor}\\
			    & \ker \N_{L/K} \ar[rr]  & & 
				\Br C 
			   }
			\end{equation}
			The top, bottom, and side squares are commutative since corestriction commutes with cohomological maps. Since we have just proved that the back square is commutative, all paths traversing from $\HH^1(\Ktilde, J[2])$ to $\Br_1 C$ are equivalent.  Hence, if we show that $\Cor\colon \HH^1(\Ktilde, J[2]) \to \HH^1(K, J[2])$ is surjective, then the front square is commutative.  
		
			We will do this by breaking the extension $\Ktilde/K$ into $3$ parts.  Let $K_1$ be the minimal field extension of $K$ such that $J(K_1)[2] \supseteq \Z/2\Z$ and let $K_2$ be the minimal field extension of $K_1$ such that $J[2] \subseteq J(K_2)$.  We will consider each part of the composition
			\[
				\HH^1(\Ktilde, J[2]) 
				\stackrel{\Cor}{\longrightarrow}
				\HH^1(K_2, J[2])
				\stackrel{\Cor}{\longrightarrow}
				\HH^1(K_1, J[2])
				\stackrel{\Cor}{\longrightarrow}
				\HH^1(K, J[2])
			\]
			separately.  
			
			Our conditions on $K_1$ imply that $[K_1:K]\mid 3$, so $\Cor\colon \HH^1(K_1, J[2]) \to \HH^1(K, J[2])$ is surjective.  Since $J[2] \subseteq J(K_2) \subseteq J(\Ktilde)$, The cohomology groups $\HH^1(\Ktilde, J[2]), \HH^1(K_2, J[2])$ are isomorphic to $(\Ktilde^{\times}/\Ktilde^{\times2})^2$ and $(K_2^{\times}/K_2^{\times2})^2$ respectively.  Under this isomorphism, Corestriction agrees with the norm map.  Thus by applying Tsen's theorem to $K_2$, we see that this map is surjective.

			It remains to prove that $\Cor\colon \HH^1(K_2, J[2]) \to \HH^1(K_1, J[2])$ is surjective.  If $K_1 = K_2$ then this is clear.  Assume that $K_2 \neq K_1$.  Fix $\tau\in G_{K_1}\setminus G_{K_2}$, and let $D_1, D_2$ be two linearly independent divisors in $J[2]$ such that $\tau(D_1) = D_2$.  Let $M = \left(\Z/2\Z\right)^4$ as an abelian group, and define a $G_{K_1}$ action on $M$ by
			\[
				\sigma\cdot(n_1, n_2, n_3, n_4) = 
					\begin{cases}
						(n_1,n_2,n_3,n_4) & \textup{if }\sigma\in G_{K_2}\\
						(n_3,n_4,n_1,n_2) & \textup{if }\sigma \notin G_{K_2}
					\end{cases}
			\]
			With this action, the map
			\[
				(n_1, n_2, n_3, n_4) \mapsto 
				\phi \colon G_{K_1} \to J[2], \quad 
				\phi(\sigma) = 
					\begin{cases}
						n_1D_1 + n_2D_2, & \textup{if }\sigma\in G_{K_2}\\
						n_3D_1 + n_4D_2, & \textup{if }\sigma\notin G_{K_2}
					\end{cases}
			\]
			gives an isomorphism of Galois modules $M\isom\Ind^{K_1}_{K_2}(J[2])$.  We fix the following isomorphism of $G_{K_1}$-modules
			\begin{align*}
				J[2] \stackrel{\sim}{\to} &\Ind^{K_1}_{K_2}(\Z/2\Z)\\
				a_1D_1 + a_2D_2 \mapsto& \phi\colon G_{K_1}\to\Z/2\Z, \quad
				\phi(\sigma) = 
					\begin{cases}
						a_1 &\textup{if } \sigma\in G_{K_2}\\
						a_2 & \textup{if }\sigma\notin G_{K_2}
					\end{cases}
			\end{align*}
			Consider the following composition
			\[
			\begin{array}{c}
				\HH^1({K_2}, J[2])\stackrel{\sim}{\longrightarrow}
				\HH^1({K_1}, \Ind^{K_1}_{K_2}(J[2])) 
				\stackrel{\textup{res}}{\longrightarrow}
				\HH^1({K_2}, M)
				\stackrel{{(\pi_1,\pi_4)}}{\longrightarrow}
				\HH^1({K_2}, (\Z/2\Z)^2)\\
				\stackrel{\sum}{\longrightarrow}
				\HH^1({K_2}, (\Z/2\Z))
				\stackrel{\sim}{\longrightarrow}
				\HH^1({K_1}, J[2]),
			\end{array}
			\]
			where the first and last map are the isomorphisms from Shapiro lemma.  By computing the map on cocycles, we see that this composition agrees with $\Cor$.  It is clear that the fourth map $\sum$ is surjective; thus, in order to prove that $\Cor$ is surjective, it suffices to show that $(\pi_1,\pi_4)\circ\textup{res}$ is surjective.  The inflation-restriction exact sequence shows that the cokernel of $\textup{res}\colon \HH^1({K_1}, \Ind^{K_1}_{K_2}(J[2])) \to\HH^1({K_2}, M)^{G_{K_1}/G_{K_2}}$ injects into $\HH^2(G_{K_1}/G_{K_2}, \Ind^{K_1}_{K_2}(J[2]))$.  Using Tate groups, we see that this is $0$, so Restriction surjects onto $\HH^1({K_2}, M)^{G_{K_1}/G_{K_2}}$.
			% \bianca{Let $G_{K_1}/G_{K_2} = \langle\tau\rangle$.  Using Tate groups, this $\HH^2$ is isomorphic to $\ker(\tau - 1)/\textup{Im} \N_{\langle\tau\rangle}$.
			% \begin{align*}
			% 	\ker(\tau - 1) &= \left\{(n_1,n_2,n_3,n_4) : (n_3,n_4,n_1,n_2) = (n_1,n_2,n_3,n_4)\right\} = \left\{(n_1, n_2, n_1, n_2)\right\}\\
			% 	\textup{Im} \N_{\langle\tau\rangle} & = \left\{\tau\cdot(n_1,n_2,n_3,n_4) + (n_1,n_2,n_3,n_4)\right\} =
			% 	\left\{(n_1 + n_3, n_2 + n_4, n_1 + n_3, n_2 + n_4)\right\} 
			% \end{align*}
			% Clearly these two sets are equal so $\HH^2 = 0$.
			% }
			We identify $\HH^1({K_2}, M)$ with $\left(K_2^{\times}/K_2^{\times2}\right)^4$; under this identification, the order $2$ quotient group $G_{K_1}/G_{K_2} = \langle \tau\rangle$ acts as 
			\[
				\tau\cdot(\alpha_1,\alpha_2,\alpha_3, \alpha_4) 
				= (\alpha_3, \alpha_4,\alpha_1,\alpha_2).
			\]
			Thus, Restriction surjects onto $\left\{(\alpha_1, \alpha_2,\alpha_1,\alpha_2) : \alpha_1,\alpha_2 \in K_2^{\times}\right\}\subset \left(K_2^\times/K_2^{\times2}\right)^4$.  From this description it is clear that $(\pi_1,\pi_4)\circ\textup{res}$ is surjective, which completes the proof.
		\end{proof}		
	
	%%%%%%%%%%%%%%%%%%%%%%%%%%%%%%%%%%%%%%%%%%%%%%%%%%%%%%%%%%%%%%%%%%%%%%%%%%%%	
	\section{An analogue of the Selmer group}\label{sec:Selmer}%%%%%%%%%%%%%%%%%
	%%%%%%%%%%%%%%%%%%%%%%%%%%%%%%%%%%%%%%%%%%%%%%%%%%%%%%%%%%%%%%%%%%%%%%%%%%%%
	
		In this section, we define the {unramified} part of $\ker N$, show that it is finite, and show that $h^{-1}\left(\Br X\right) \subseteq \left(\ker N\right)_{S\textup{-unr}}$.  In order to define when an element $\ell\in L$ is {unramified}, we first fix some notation.
		
		The degree $4$ $K$-algebra $L$ can be decomposed as a direct product of fields $L = \prod_i L_i$.  Each $L_i$ has transcendence degree $1$, so each $L_i$ is the function field of a smooth projective geometrically integral curve, $Z_i$.  Let $Z = \bigcup Z_i$; then $\kk(Z) = L$.  Since $L$ is a degree $4$ algebra over $K$, we have a degree $4$ map $\varpi\colon Z\to W$.  Note that we also have a surjective map $Z\to \overline{V(y)} \subseteq X$, where $\overline{V(y)}$ denotes the Zariski closure of $V(y)\subseteq C$ in $X$.  The $4:1$ cover $\varpi$ agrees with the composition $Z \to \overline{V(y)} \stackrel{\pi}{\to} W$.
	
		Let $S\subset W$ be a finite set of points that contains all of the places where the model $y^2 = f(x)$ has bad reduction.  By definition of $Z$ and $\varpi$, S contains all points of $W$ where the map $\varpi\colon Z\to W$ fails to be smooth.  We say an element $\ell \in L^{\times}$ is \defi{S-unramified} if for all $t \in W\setminus S$ and for all $P,Q\in Z_t$, we have $v_P(\ell) \equiv v_Q(\ell) \pmod 2$.  Note that $\ell$ is $S$-unramified if and only if $\ell\ell'$ is $S$-unramified for any $\ell' \in L^{\times2}K^{\times}$, thus we may consider $S$-unramified elements of the quotient $L^{\times}/L^{\times2}K{\times}$.  We denote the subgroup of $\ker \N$ that consists of all $S$-unramified elements by $\left(\ker \N\right)_{S\textup{-unr}}$.
			
		\begin{thm}\label{thm:finiteness}
			The group $\left(\ker \N\right)_{S\textup{-unr}}$ is finite.  In addition, $h^{-1}\left(\Br X\right) \subseteq (\ker \N)_{S\textup{-unr}}$, so $h^{-1}\left(\Br X\right)$ is finite.
		\end{thm}
		\begin{proof}
			
			We first prove the containment $h^{-1}\left(\Br X\right) \subseteq (\ker \N)_{S\textup{-unr}}$; then we will prove finiteness.
			
			The purity theorem states that
			\[
				0 \to \Br X \to \Br C \to 
				\bigoplus_{\substack{ V \in X^{(1)}\\V \textup{ vertical }}} 
				\HH^1(\kappa(V), \Q/\Z)
			\]
			is exact.  Let $\ell \in  (\ker \N)\setminus(\ker \N)_{S\textup{-unr}}$.  Then there is a point $t_0 \in W\setminus S$ and $P,Q \in Z_{t_0}$ such that $X_{t_0}$ is a smooth genus $1$ curve and $v_P(\ell) \not\equiv v_Q(\ell) \pmod 2$.  We will show that $\partial_{X_{t_0}}(\Cor(\ell, x-\alpha))$ is nonzero, thus showing that $\ell \notin h^{-1}\left(\Br X\right)$.
			
			Due to the smoothness of $X_{t_0}$, $Z_{t_0} = \{ P_i(t_0) = (\alpha_i(t_0), 0) : i = 1,2,3,4\}$.  Since $\N(\ell) \in K^{\times2}$, after appropriate scaling by an element in $L^{\times2}K^{\times}$, we may assume that $v_{P_1}(\ell) = v_{P_2}(\ell) = 1$ and $v_{P_3}(\ell) = v_{P_4}(\ell) = 0$.  Then, by applying the analogue of~\eqref{diag:purity-cores} for vertical divisors, we have
			\[
				\partial_{X_{t_0}}(\Cor(\ell, x-\alpha)) = (x - \alpha_1(t_0))(x - \alpha_2(t_0)) \in \kappa(X_{t_0})^{\times}/\kappa(X_{t_0})^{\times2}.
			\]  
			The divisor $P_1(t_0) - P_2(t_0)$ is a non-trivial $2$ torsion element of $\Pic(X_{t_0})$ so $\frac{x - \alpha_1(t_0)}{x - \alpha_2(t_0)} $ is not a square.  This completes the proof of the containment.
			
			Now we will prove that $\left(\ker \N\right)_{S\textup{-unr}}$ is finite.  Let $\ell\in L^{\times}$ be a representative of an element in $\left(\ker \N\right)_{S\textup{-unr}}$.  Let $t_1, \ldots, t_r \in W\setminus S$ be such that $v_P(\ell) \equiv 1 \pmod 2$ for all $P\in \varpi^{-1}\left(t_1\cup\cdots\cup t_r\right)$ and let $s\in S$.  Since $\Jac(W)$ is divisible, there exists some $D\in \Jac(W)$ such that 
			\[
				t_1 + t_2 + \cdots t_r - r\cdot s \sim 2D;
			\]
			Let $\gamma \in K$ be an element that gives this equivalence, i.e. $\divv(\gamma) = t_1 + t_2 + \cdots + t_r - r\cdot s - 2D$.  By replacing $\ell$ with $\ell\gamma$, we may assume that for all $P \in Z\setminus \varpi^{-1}(S)$ we have $v_P(\ell) \equiv 0 \pmod 2$.  
			
			For all $P, Q \in \varpi^{-1}(S)$, let $\ell_{P,Q} \in L^{\times}$ be such that $\divv(\ell_{P, Q}) = P - Q + 2D$ for some degree $0$ divisor $D$.  After multiplying $\ell$ by a suitable finite number of $\ell_{P,Q}$, then we obtain an element $\ell'\in L^{\times2}$ such that $\divv(\ell') = 2D$.  Thus, the class of $\ell'$ in $L^{\times}/L^{\times2}$ is determined by an element of $\Jac(Z)[2]$.  Since $\Jac(Z)[2]$ is finite of over $2^{2g(Z)}$, this shows that $\left(\ker \N\right)_{S\textup{-unr}}$
		\end{proof}

%%%%%%%%%%%%%%%%%%%%%%%%%%%%%%%%%%%%%%%%%%%%%%%%%%%%%%%%%%%%%%%%%%%%%%%%%%%%%%%%
%%%%%%%%%%%%%%%%%%				Bibliography			%%%%%%%%%%%%%%%%%%%%%%%%
%%%%%%%%%%%%%%%%%%%%%%%%%%%%%%%%%%%%%%%%%%%%%%%%%%%%%%%%%%%%%%%%%%%%%%%%%%%%%%%%

\end{document}